%% file: main.tex
\documentclass[12pt,oneside]{amsart}

\usepackage{setspace,amssymb,amsmath,mathrsfs,enumitem,textcomp,amsbsy,url}
\usepackage[all,cmtip]{xy}

\usepackage{datetime}
\renewcommand{\dateseparator}{-}
\renewcommand{\today}{\the\year \dateseparator \twodigit\month
\dateseparator \twodigit\day}

\usepackage
[pdfauthor={Stergios M. Antonakoudis},
 bookmarks=false,
 hyperfootnotes=false]
{hyperref}

\usepackage{needspace,graphicx,extpfeil,pdfpages}
\usepackage{fancyhdr}
\fancypagestyle{plain}
{
  \fancyhf{}               % clear all header and footer fields
  \fancyhead[C]{\thepage}  % except for the center
}

\newcounter{restatecount}
\theoremstyle{plain}
\newtheorem{thm}{Theorem}[section]

\newtheorem{thm-ref}{Theorem}

\newtheorem{lem}[thm]{Lemma}
\newtheorem{cor}[thm]{Corollary}

\theoremstyle{definition}
\newtheorem{defn}[thm]{Definition}

\theoremstyle{remark}
\newtheorem*{remark}{Remark}

\numberwithin{equation}{section}
\addtocontents{toc}{\protect\setcounter{tocdepth}{1}}

\newcommand{\M}{{\mathcal M}}
\newcommand{\T}{{\mathcal T}}

\newcommand{\C}{{\mathbb C}}
\newcommand{\R}{{\mathbb R}}
\title{Isometric disks are holomorphic}

\author{Stergios M. Antonakoudis} 

\thanks{Department of Pure Mathematics and Mathematical Statistics, University of Cambridge, UK~\\Stergios M. Antonakoudis, stergios@dpmms.cam.ac.uk}

\begin{document}

  \maketitle
  \input{abstract}
  \tableofcontents

  \input{introduction}
  \input{background}
  \input{disks}
  %\input{questions}

\input{main.bbl}
\end{document}

%% file: abstract.tex
% ----------------------------------------------------------------

\begin{abstract}
  This paper shows that every totally-geodesic isometry from the unit disk to a finite-dimensional
  Teichm\"uller space for the intrinsic Kobayashi metric is either holomorphic or anti-holomorphic; in
  particular, it is a Teichm\"uller disk. Additionally, a similar result is proved for a large class of
  \textit{disk-rigid} domains, which includes strictly convex bounded domains, as well as Teichm\"uller
  spaces.
\end{abstract}

% ----------------------------------------------------------------

%%% Local Variables:
%%% TeX-master: "main"
%%% End:

% ----------------------------------------------------------------

%% file: introduction.tex
% ----------------------------------------------------------------

\section{Introduction}\label{sec:intro} %(fill-column 100)
Let $\mathbb{CH}^1$ denote the unit disk $\Delta=\{~z\in \C : |z| < 1~\}$ equipped with its
Poincar\'e metric $|dz|/(1-|z|^2)$ and let $\T_{g,n}$ denote a finite-dimensional Teichm\"uller
space equipped with its \textit{intrinsic Kobayashi metric}, which is the largest (Finsler) metric
such that every holomorphic map $f: \mathbb{CH}^1 \rightarrow \T_{g,n}$ is \textit{non-expanding}:
$||df||\leq 1$.

An important feature of the Kobayashi metric of $\T_{g,n}$ is that every holomorphic map
$f: \mathbb{CH}^1 \rightarrow \T_{g,n}$, for which $df$ is an isometry on tangent spaces, is
\textit{totally geodesic}: it sends real geodesics to real geodesics preserving their
length. Moreover, there are such isometries through every point in every direction, known as
Teichm\"uller disks.~\cite{Kra:survey}

\subsection*{Holomorphic rigidity for Teichm\"uller spaces} Our main result in this paper is the following:

\begin{thm}\label{thm:disks:intro}
  Let $\T_{g,n}$ be a finite-dimensional Teichm\"uller space equipped with its intrinsic Kobayashi
  metric. Every totally geodesic isometry $f: \mathbb{CH}^1 \hookrightarrow \T_{g,n}$ is either
  holomorphic or anti-holomorphic; in particular, it is a Teichm\"uller disk.
 \end{thm}

 \begin{remark}
   Theorem~\ref{thm:disks:intro} settles a long standing problem in Teichm\"uller theory~\footnote{See
     problem 5.3 in \cite{Fletcher:Markovic:survey}.}.
 \end{remark}

 The proof is geometric and rests on the idea of \textit{complexification}; see
 \S~\ref{sec:disks}. Informally, the theorem shows that the intrinsic Kobayashi metric of $\T_{g,n}$
 determines its natural structure as a complex manifold.

 As a corollary, we obtain the following general result about Teichm\"uller spaces.
\begin{cor}\label{cor:intro}
  Let $\T_{g,n}$, $\T_{h,m}$ be two finite-dimensional Teichm\"uller spaces equipped with their intrinsic
  Kobayashi metric. Every totally geodesic isometry $f: \T_{g,n} \hookrightarrow \T_{h,m}$ is either
  holomorphic or anti-holomorphic.
\end{cor}
We note that there are, indeed, many holomorphic isometries $f: \T_{g,n} \hookrightarrow \T_{h,m}$
between Teichm\"uller spaces for their Kobayashi metric.~\cite{Kra:survey}
% For a related application of Corollary~\ref{cor:intro}, see ~\cite{Antonakoudis:isometric:geometric}.

\subsection*{Holomorphic rigidity for convex domains} In addition to Theorem~\ref{thm:disks:intro},
we prove a similar result for a large class of \textit{disk-rigid} domains, which include strictly
convex bounded domains, as well as Teichm\"uller spaces. We discuss the general statement in
\S~\ref{sec:disk-rigid}; as a special case, we obtain:

\begin{thm}\label{thm:disks:other}
  Let $\mathcal{B}_1$, $\mathcal{B}_2$ be two strictly convex bounded domains equipped with their
  intrinsic Kobayashi metric. Every totally geodesic isometry
  $f: \mathcal{B}_1 \hookrightarrow \mathcal{B}_2$ is either holomorphic or anti-holomorphic.
\end{thm}

This result need not be true for general convex domains. For example, the diagonal map
$\delta(z)=(z, z)$ is a totally-real embedding
$\delta: \mathbb{CH}^1 \hookrightarrow \mathbb{CH}^1\times\overline{\mathbb{CH}^1}$, which is a
totally geodesic isometry for the Kobayashi metric. In particular, the result is not true for
bounded symmetric domains with rank two or more.~\\

% ----------------------------------------------------------------
\subsection*{Notes and References}~\\
% ----------------------------------------------------------------
For an introduction to Teichm\"uller spaces and the Kobayashi metric on complex manifolds, we refer
to \cite{Gardiner:Lakic:book},\cite{Hubbard:book:T1} and ~\cite{Kobayashi:book:hyperbolic},
respectively.

H. L. Royden proved that the Kobayashi metric of $\T_{g,n}$ coincides with its classical
Teichm\"uller metric.~\cite{Royden:metric}. When $\text{dim}_{\C}\T_{g,n} = 1$, we can identify
$\T_{g,n}$ equipped with its Kobayashi-Teichm\"uller metric with the unit disk $\mathbb{CH}^1$
equipped with its Poincar\'e metric. In particular, the first instance of
Theorem~\ref{thm:disks:intro} is implicit in the natural isomorphism
$\text{Aut}(\mathbb{CH}^1) \cong \text{Isom}^{+}(\mathbb{CH}^1)$ between the group of holomorphic
automorphisms of the unit disk and the group of orientation-preserving isometries of its Poincar\'e
metric.

There is a natural action of $\text{SL}_2(\mathbb{R})$ on the sphere bundle of unit-area quadratic
differentials $Q_1\T_{g,n}$ over $\T_{g,n}$, so that every orbit projects to a holomorphic totally
geodesic isometry
$\mathbb{CH}^1 \cong \text{SO}_2(\mathbb{R})\setminus\text{SL}_2(\mathbb{R}) \hookrightarrow
\T_{g,n}$,
which is known as a \textit{Teichm\"uller disk}.  It is a classical result that every
\textit{holomorphic} isometry $\mathbb{CH}^1 \hookrightarrow \T_{g,n}$ into a finite-dimensional
Teichm\"uller space is a Teichm\"uller disk. However, neither this result, nor
Theorem~\ref{thm:disks:intro} remain true for infinite-dimensional Teichm\"uller spaces.
% In addition, it follows from recent deep results that the image of every Teichm\"uller disk to the
% moduli space of curves $\mathcal{M}_{g,n}$ has the remarkable property that its closure is an algebraic
% subvariety of $\mathcal{M}_{g,n}$.~\cite{Filip:hodge:affine},~\cite{Eskin:Mirzakhani:higherg}

A complex analytic proof that totally geodesic disks are holomorphic for strictly convex domains
with $\text{C}^3$-smooth boundary appears in~\cite{Harish:Herve:geodesicdisks}.
Theorem~\ref{thm:disks:other} gives an optimal result for maps between convex domains. We also note
that Teichm\"uller spaces $\T_{g,n}\subset \C^{3g-3+n}$ cannot be realised as convex domains.

% ----------------------------------------------------------------

%%% Local Variables:
%%% TeX-master: "main"
%%% End:

% ----------------------------------------------------------------

%% file: background.tex
% ----------------------------------------------------------------

\addtocontents{toc}{\protect\setcounter{tocdepth}{1}}

% ----------------------------------------------------------------
\section{Preliminary results}\label{sec:prelim}
% ----------------------------------------------------------------

\subsection*{The Kobayashi metric}\label{sec:kobayashi-metric}~\cite{Kobayashi:book:hyperbolic}
Let $\mathcal{B}\subset\C^{N}$ be a bounded domain. The intrinsic Kobayashi metric of $\mathcal{B}$
is the \textit{largest} complex Finsler metric such that every holomorphic map
$f: \mathbb{CH}^1 \rightarrow \mathcal{B}$ is non-expanding: $||df||_{\mathcal{B}}\leq 1$. It
determines both a family of norms $||\cdot||_{\mathcal{B}}$ on tangent spaces and a distance
function $d_{\mathcal{B}}(\cdot,\cdot)$ on pairs of points.

By Schwarz's lemma, every holomorphic map $f:\mathbb{CH}^1\rightarrow \mathbb{CH}^1$ is
non-expanding. The Kobayashi metric provides a natural generalisation - it has the fundamental
property that every holomorphic map between complex domains is non-expanding. In particular, a
holomorphic automorphism is always an isometry and the Kobayashi metric of a complex domain depends
only on its structure as a complex manifold.

\subsection*{Examples}
  \begin{enumerate}[leftmargin=2em]
  \item $\mathbb{CH}^1\cong \{~z\in \C : |z| < 1~\}$ with its Poincar\'e metric $|dz|/(1-|z|^2)$
    coincides with the Kobayashi metric, by Schwarz's lemma. More generally, the Kobayashi metric on
    the unit ball $\mathbb{CH}^N \cong \{~(z_i)_{i=1}^N\in\C^N ~:~ \sum_{i=1}^N|z_i|^2 < 1~\}$
    coincides
    with its complete invariant (Ka\"ehler) metric of constant holomorphic curvature -4.~\\
  \item The Kobayashi metric of the bi-disk $\mathbb{CH}^1\times\overline{\mathbb{CH}^1}$ is the
    maximum metric of the two factors. It is a complex Finsler metric, which is not Hermitian. The
    distance function is given by
    $d_{\mathbb{CH}^1\times \overline{\mathbb{CH}^1}} ((z_1,z_2), (w_1,w_2)) =
    \max\{d_{\mathbb{CH}^1}(z_1,w_1), d_{\mathbb{CH}^1}(z_2,w_2)\}$
    for all points $(z_1,z_2),(w_1,w_2) \in \mathbb{CH}^1\times\overline{\mathbb{CH}^1}$.~\\
  \item The Kobayashi metric of $\T_{g,n}$ coincides with the classical Teichm\"uller metric, which
    endows $\T_{g,n}$ with the structure of a complete geodesic metric space. We discuss this
    example in more detail below.
  \end{enumerate}
% ----------------------------------------------------------------
  %\begin{examples}~\\
  % \noindent 1. $\mathbb{CH}^1\cong \{~z\in \C : |z| < 1~\}$ with its Poincar\'e metric
  % $|dz|/(1-|z|^2)$ coincides with the Kobayashi metric, by Schwarz's lemma. More generally, the
  % Kobayashi metric on the unit ball
  % $\mathbb{CH}^N \cong \{~(z_1,\ldots ,z_N)\in\C^N ~:~ |z_1|^2 + \ldots + |z_N|^2 < 1~\}$ coincides
  % with its unique (complete) invariant Ka\"ehler metric of constant holomorphic curvature -4.~\\
  % ----------------------------------------------------------------
  % \noindent 2. The Kobayashi metric of the bi-disk $\mathbb{CH}^1\times\overline{\mathbb{CH}^1}$ is
  % the maximum of the two metrics from the factors. It is a complex Finsler metric, which is not
  % Hermitian. The distance function is given by
  % $d_{\mathbb{CH}^1\times \overline{\mathbb{CH}^1}} ((z_1,z_2), (w_1,w_2)) =
  % \max\{d_{\mathbb{CH}^1}(z_1,w_1), d_{\mathbb{CH}^1}(z_2,w_2)\}$
  % for all points $(z_1,z_2),(w_1,w_2) \in \mathbb{CH}^1\times\overline{\mathbb{CH}^1}$.~\\
  % ----------------------------------------------------------------  
  % \noindent 3. The Kobayashi metric of $\T_{g,n}$ coincides with the classical Teichm\"uller metric,
  % which endows $\T_{g,n}$ with the structure of a complete geodesic metric space. We discuss this
  % example in more detail below.
  %\end{examples}
% ----------------------------------------------------------------
\subsection*{Complex geodesics}\label{sec:complex-geodesics}
A holomorphic (or anti-holomorphic) map $\gamma_{\C}: \mathbb{CH}^1 \rightarrow \mathcal{B}$ is
locally distance preserving for the Kobayashi metric if and only if it is a totally geodesic
isometry: $\gamma_{\C}$ sends real geodesics to real geodesics preserving their length. We call such
a map a \textit{complex geodesic}. We note that in this case, for every
$\theta \in\mathbb{R}/2\pi\mathbb{Z}$, the map given by
$\gamma(t) = \gamma_{\C} (e^{i\theta}\text{tanh}(t))$, for $t \in \R$, defines a complete,
unit-speed, real geodesic line in $\mathcal{B}$. When it is clear from the context, we will often
identify real and complex geodesics with their image in $\mathcal{B}$.~\\
% ----------------------------------------------------------------

\subsection*{Teichm\"uller space}~\cite{Gardiner:Lakic:book},~\cite{Hubbard:book:T1} Let
$\Sigma_{g,n}$ be a connected, oriented surface of genus $g$ and $n$ punctures and $\T_{g,n}$ denote
the Teichm\"uller space of Riemann surfaces marked by $\Sigma_{g,n}$. A point in $\T_{g,n}$ is
specified by an equivalence class\footnote{Two marked Riemann surfaces
  $ \phi: \Sigma_{g,n} \rightarrow X$, $\psi: \Sigma_{g,n}\rightarrow Y$ are equivalent if
  $\psi \circ {\phi}^{-1}: X \rightarrow Y$ is isotopic to a holomorphic bijection.} of orientation
preserving homeomorphisms $\phi: \Sigma_{g,n} \rightarrow X$, where $X$ is a Riemann surface of finite type.

Teichm\"uller space $\T_{g,n}$ is the orbifold universal cover of the moduli space of Riemann
surfaces $\M_{g,n}$ and is naturally a complex manifold with dimension $3g-3+n$. It is known that
Teichm\"uller space can be realized as a contractible bounded domain of holomorphy
$\T_{g,n} \subset\C^{3g-3+n}$ by the Bers embeddings.~\cite{Bers:ts:survey}
% ----------------------------------------------------------------

\subsection*{Teichm\"uller metric}
For each $X \in \T_{g,n}$, we let $Q(X)$ denote the space of holomorphic quadratic differentials
$q=q(z)(dz)^2$ on $X$ with finite total mass: $ ||q||_{1} = \int_{X} |q(z)||dz|^2 < +\infty$, which
means that $q$ has at worse simple poles at the punctures of $X$. The tangent and cotangent spaces
to Teichm\"uller space at $X\in \T_{g,n}$ are described in terms of the natural pairing
$ (q,\mu) \mapsto \int_{X} q\mu$ between the space $Q(X)$ and the space $M(X)$ of
$L^{\infty}$-measurable Beltrami differentials on $X$; in particular, the tangent $T_{X} \T_{g,n}$
and cotangent $T_{X}^{*} \T_{g,n}$ spaces are naturally isomorphic to $M(X)/Q(X)^{\perp}$ and
$Q(X)$, respectively.

The Teichm\"uller-Kobayashi metric on $\T_{g,n}$ is given by norm duality on the tangent space
$T_{X}\T_{g,n}$ from the norm $||q||_{1} = \int_{X} |q|$ on the cotangent space $Q(X)$ at $X$. The
corresponding distance function is given by the formula
$d_{\T_{g,n}}(X,Y) = \inf \frac{1}{2} \log K(\phi)$ and measures the minimal dilatation $K(\phi)$ of
a quasiconformal map $\phi: X \rightarrow Y$ respecting their markings. 

The Teichm\"uller metric is complete and coincides with the Kobayashi metric of $\T_{g,n}$ as a
complex manifold.~\cite{Royden:metric} In particular, it has the remarkable property that every
holomorphic map $f : \mathbb{CH}^1 \rightarrow \T_{g,n}$ is non-expanding:
$||df||_{\T_{g,n}} \leq 1$.
% ----------------------------------------------------------------
% We denote by $Q\T_{g,n} \cong T^{*}\T_{g,n}$ the complex vector-bundle of holomorphic quadratic
% differentials over $\T_{g,n}$ and by $Q_1\T_{g,n}$ the associated sphere-bundle of quadratic
% differentials with unit mass. There is a natural norm-preserving action of $\text{SL}_2(\mathbb{R})$
% on $Q\T_{g,n}$, with the diagonal matrices giving the (co-)geodesic flow.
% ----------------------------------------------------------------
% For each $(X,q)\in Q_1\T_{g,n}$, the orbit $\text{SL}_2(\mathbb{R}) \cdot (X,q) \subset Q_1\T_{g,n}$
% induces a holomorphic totally geodesic isometry
% \[\mathbb{CH}^1 \cong
%   \text{SO}_2(\mathbb{R})\setminus\text{SL}_2(\mathbb{R})
%   \hookrightarrow \T_{g,n}\]
% which we refer to as the \textit{Teichm\"uller disk} generated by
% $(X,q)$.
% ----------------------------------------------------------------
% Let $X \in \T_{g,n}$ and let $q\in Q(X)$ generate a real Teichm\"uller geodesic $\gamma$ with
% $\gamma(0)=X$. The geodesic ray $\gamma$ extends uniquely to a holomorphic totally geodesic isometry
% $\gamma_{\C}: \Delta \cong \mathbb{CH}^1 \hookrightarrow \T_{g,n}$ satisfying
% $\gamma(t)=\gamma_{\C}(\tanh(t))$ for $t \in \R$; the Teichm\"uller geodesic generated by the
% quadratic differential $e^{i\theta}q \in Q(X)$, with $\theta \in\mathbb{R}/2\pi\mathbb{Z}$, is given
% by the map $t \mapsto \gamma_{\C} (e^{-i\theta}\text{tanh}(t))$, $t \in \R$.
% ----------------------------------------------------------------
\subsection*{Holomorphic disks} We summarise below the main results about holomorphic disks in
Teichm\"uller space which we shall employ in the proof of Theorem\ref{thm:disks:intro}.

Complex geodesics in Teichm\"uller space are abundant: there is one through every point in
$\T_{g,n}$ in every complex direction, classically known as Teichm\"uller disks.

Every complex geodesic $\gamma_{\C}: \mathbb{CH}^1 \hookrightarrow \T_{g,n}$ gives rise to a
unit-speed real geodesic $\gamma : \R \hookrightarrow \T_{g,n}$ by
$\gamma(t) = \gamma_{\C} (\text{tanh}(t))$, for $t \in \R$. Conversely, every unit-speed real
geodesic $\gamma: \R \hookrightarrow \T_{g,n}$ extends \textit{uniquely} to a complex geodesic
$\gamma_{\C}: \Delta\cong\mathbb{CH}^1 \hookrightarrow \T_{g,n}$ such that
$\gamma(t) = \gamma_{\C} (\text{tanh}(t))$, for $t \in \R$.

The following result characterises the holomorphic disks in Teichm\"uller space which are complex
geodesics for the Kobayashi metric. See~\cite{Earle:Kra:Krushkal}, for a simple proof based on
Slodkowski's theorem~\cite{Slodkowski:motions}.

\begin{thm}\label{thm:slodkowski}
  Let $f:\Delta\cong\mathbb{CH}^1 \rightarrow \T_{g,n}$ be a holomorphic map with
  $||f'(0)||_{\T_{g,n}}=1$, then $f$ is a totally geodesic isometry for the Kobayashi metric. In
  particular, it is a Teichm\"uller disk.
\end{thm}
% ----------------------------------------------------------------

The following important result shows that there are no non-trivial holomorphic families of
\textit{essentially proper} holomorphic disks in Teichm\"uller space. It is a consequence of
Sullivan's rigidity theorem ~\cite{Sullivan:linefield}; see~\cite{Tanigawa:holomap} for a proof and
~\cite{McMullen:fermat},~\cite{Shiga:regular} for further applications and related ideas.

\begin{thm}\label{thm:sullivan}
  Let $\{f_t\}_{t\in\Delta}$ be a holomorphic family of holomorphic maps
  $f_t : \Delta\cong\mathbb{CH}^1 \rightarrow \T_{g,n}$. If $\partial\Delta \setminus B_{f_0}$ has
  positive (Lebesgue) measure, where $B_{f_0}$ denotes the set of \textit{bounded rays} of $f_0$,
  ie.
  $B_{f_0} = \{~ e^{i\theta}\in\partial\Delta ~:~ \sup_{t\in[0,1)}d_{\T_{g,n}}(f_0(0),f_0(te^{i\theta})) < +\infty ~\}$,
then the family is \textit{trivial}: $f_t = f_0$ for all $t\in\Delta$.
\end{thm}

There are other bounded domains that satisfy the same properties about holomorphic disks as
above. We will discuss this class of \textit{disk-rigid} domains and formulate a generalisation of
Theorem~\ref{thm:disks:intro} in \S~\ref{sec:disk-rigid}.

% ----------------------------------------------------------------

% ----------------------------------------------------------------

%%% Local Variables:
%%% TeX-master: "main"
%%% End:

% ----------------------------------------------------------------

%% file: disks.tex
% ----------------------------------------------------------------

\section{Holomorphic rigidity for Teichm\"uller spaces}\label{sec:disks}                                                                              

In this section we prove:
\begin{thm}\label{thm:disks}
  Every totally geodesic isometry
  $f: \mathbb{CH}^1 \hookrightarrow \T_{g,n}$ for the Kobayashi metric
  is either holomorphic or anti-holomorphic. In particular, it is a
  Teichm\"uller disk.
\end{thm}
The proof of the theorem uses the idea of \textit{complexification}
and leverages the following two facts.  Firstly, a complete real
geodesic in $\T_{g,n}$ is contained in a unique holomorphic
Teichm\"uller disk; and secondly, a holomorphic family
$\{f_t\}_{t\in\Delta}$ of \textit{essentially proper} holomorphic maps
$f_t : \mathbb{CH}^1 \rightarrow \T_{g,n}$ is \textit{trivial}:
$f_t = f_0$ for $t\in\Delta$ (Sullivan's rigidity theorem, see
~\cite{Tanigawa:holomap} for a precise statement and proof).

\subsection*{Outline of the proof} 
Let $\gamma \subset \mathbb{CH}^1$ be a complete real geodesic and
denote by
$\gamma_{\C} \subset \mathbb{CH}^1\times\overline{\mathbb{CH}^1}$ its
\textit{maximal} holomorphic extension to the bi-disk. We note that
$\gamma_{\C} \cong \mathbb{CH}^1$ and we define $F|_{\gamma_{\C}}$ to
be the \textit{unique} holomorphic extension of $f|_{\gamma}$, which
is a Teichm\"uller disk.

Applying this construction to all (real) geodesics in $\mathbb{CH}^1$,
we will deduce that $f: \mathbb{CH}^1 \rightarrow \T_{g,n}$ extends to
a \textit{holomorphic} map
$F:\mathbb{CH}^1\times\overline{\mathbb{CH}^1}\rightarrow \T_{g,n}$
such that $f(z)=F(z,z)$ for $z\in \Delta \cong \mathbb{CH}^1$. Using
that $f$ is totally geodesic, we will show that $F$ is
\textit{essentially} proper and hence, by Sullivan's rigidity theorem,
we will conclude that either $F(z,w) = F(z,z)$ or $F(z,w)=F(w,w)$, for
all $(z,w)\in \mathbb{CH}^1\times\overline{\mathbb{CH}^1}$.\qed
~\\
% ----------------------------------------------------------------

\large
\[
  \xymatrix{
    \mathbb{CH}^1 \times \overline{\mathbb{CH}^1} \ar@{->}^F[rd] \\
    \mathbb{CH}^1 \ar@{^{(}->}^{\delta}[u] \ar@{^{(}->}^f[r] &\T_{g,n} }
\]
~\\
\normalsize
% ----------------------------------------------------------------

We start with some preliminary constructions.
\subsection*{The totally real diagonal}
Let $\overline{\mathbb{CH}^1}$ be the complex hyperbolic line with its
conjugate complex structure. The identity map is a \textit{canonical}
anti-holomorphic isomorphism
$\mathbb{CH}^{1}\cong\overline{\mathbb{CH}^1}$ and its graph is a
totally real embedding
$\delta: \mathbb{CH}^{1} \hookrightarrow
\mathbb{CH}^{1}\times\overline{\mathbb{CH}^1}$,
given by $\delta(z)=(z,z)$ for $z\in \Delta\cong \mathbb{CH}^{1}$. We
call $\delta(\mathbb{CH}^{1})$ the \textit{totally real diagonal}.
% ----------------------------------------------------------------

\subsection*{Geodesics and graphs of reflections}
Let $\mathcal{G}$ denote the set of all real, unoriented, complete
geodesics $\gamma \subset \mathbb{CH}^{1}$. In order to describe their
\textit{maximal} holomorphic extensions
$\gamma_{\C} \subset \mathbb{CH}^1\times\overline{\mathbb{CH}^1}$,
such that $\gamma_{\C} \cap \delta(\mathbb{CH}^1) = \delta(\gamma)$,
it is convenient to parametrize $\mathcal{G}$ in terms of the set
$\mathcal{R}$ of hyperbolic reflections of $\mathbb{CH}^1$ - or
equivalently, the set of anti-holomorphic involutions of
$\mathbb{CH}^1$. The map that associates a reflection
$r\in \mathcal{R}$ with the set
$\gamma = \text{Fix}(r) \subset \mathbb{CH}^{1}$ of its fixed points
gives a bijection between $\mathcal{R}$ and $\mathcal{G}$.

Let $r\in\mathcal{R}$ and denote its graph by
$\Gamma_r\subset\mathbb{CH}^{1}\times\overline{\mathbb{CH}^{1}}$;
there is a natural holomorphic isomorphism
$\mathbb{CH}^1 \cong \Gamma_r$, given by $z \mapsto (z,r(z))$ for
$z\in \Delta \cong \mathbb{CH}^1$. We note that $\Gamma_r$ is the
\textit{maximal} holomorphic extension $\gamma_{\C}$ of the geodesic
$\gamma = \text{Fix}(r)$ to the bi-disk and it is \textit{uniquely}
determined by the property
$\gamma_{\C} \cap \delta(\mathbb{CH}^1) = \delta(\gamma)$.

% ----------------------------------------------------------------
\subsection*{The foliation by graphs of reflections}
The union of the graphs of reflections
$\bigcup_{r\in\mathcal{R}}\Gamma_r$ gives rise to a (singular)
foliation of $\mathbb{CH}^{1}\times\overline{\mathbb{CH}^{1}}$ with
holomorphic leaves $\Gamma_r$ parametrized by the set
$\mathcal{R}$. We have
$\displaystyle\Gamma_r \cap \delta(\mathbb{CH}^{1}) =
\delta(\text{Fix}(r))$ for all $r\in\mathcal{R}$, and
\begin{equation}\label{eq:leaves}
\displaystyle\Gamma_r \cap \Gamma_s = \delta(\text{Fix}(r) \cap \text{Fix}(s))
\end{equation}
which is either empty or a single point for all $r,s \in \mathcal{R}$
with $r \neq s$. In particular, the foliation is smooth in the
complement of the totally real diagonal $\delta(\mathbb{CH}^{1})$.

We emphasize that the following simple observation plays a key role in
the proof of the theorem. For all $r\in\mathcal{R}$:
\begin{equation}
  \label{eq:flip}
  (z,w) \in \Gamma_r \iff (w,z) \in \Gamma_r 
\end{equation}
% ----------------------------------------------------------------
\subsection*{Geodesics and the Klein model}\label{sec:klein-model}
The Klein model gives a real-analytic identification
$\mathbb{CH}^1\cong\mathbb{RH}^2\subset\R^2$ with an open disk in
$\R^2$. It has the nice property that the hyperbolic geodesics are
affine straight lines intersecting the disk.~\cite{Ratcliffe:book}
\begin{remark}
  The holomorphic foliation by graphs of reflections defines a
  \textit{canonical} complex structure in a neighborhood of the zero
  section of the tangent bundle of $\mathbb{RH}^2$.
\end{remark}
The description of geodesics in the Klein model is convenient in the
light of the following theorem of S.~Bernstein.

\begin{thm}\label{thm:bernstein}(\cite{Ahiezer:Ronkin:Bernstein};~S.~Bernstein)
  Let $M$ be a complex manifold, $f: [0,1]^2 \rightarrow M$ a map from
  the square $[0,1]^2 \subset \R^2$ into $M$ and $E\subset \C$ an
  ellipse with foci at $0,1$. If there are holomorphic maps
  $F_{\ell} : E \rightarrow M$ such that
  $F_{\ell}|_{[0,1]} = f|_{\ell}$, for all vertical and horizontal
  slices $\ell\cong [0,1]$ of $[0,1]^2$, then $f$ has a unique
  holomorphic extension in a neighborhood of $[0,1]^2$ in $\C^2$.
\end{thm}
We use this to prove:
\begin{lem}\label{lem-real-analytic}
  Every totally geodesic isometry
  $f:\mathbb{CH}^{1} \hookrightarrow \T_{g,n}$ admits a unique
  holomorphic extension in a neighborhood of the totally real diagonal
  $\delta(\mathbb{CH}^{1})\subset\mathbb{CH}^{1}\times\overline{\mathbb{CH}^{1}}$.
\end{lem}
\begin{proof}[Proof of~\ref{lem-real-analytic}]
  Using the fact that analyticity is a local property and the
  description of geodesics in the Klein model of $\mathbb{RH}^2$, we
  can assume - without loss of generality - that the map $f$ is
  defined in a neighborhood of the unit square $[0,1]^2$ in $\R^2$ and
  has the property that its restriction on every horizontal and
  vertical line segment $\ell \cong [0,1]$ is a real-analytic
  parametrization of a Teichm\"uller geodesic segment. Moreover, we
  can also assume that the lengths of all these segments, measured in
  the Teichm\"uller metric, are uniformly bounded from above and from
  below away from zero.

  Since every segment of a Teichm\"uller geodesic extends to a
  (holomorphic) Teichm\"uller disk in $\T_{g,n}$, there exists an
  ellipse $E\subset\C$ with foci at $0$,$1$ such that the restrictions
  $f|_{\ell}$ extend to holomorphic maps
  $F_{\ell}: E \rightarrow \T_{g,n}$ for all horizontal and vertical
  line segments $\ell\cong [0,1]$ of $[0,1]^2$.  Hence, the proof of
  the lemma follows from Theorem~\ref{thm:bernstein}.
\end{proof}
\begin{remark}
  See~\cite{Shiffman:separate:analyticity}, for a strongest result regarding separate analyticity.
\end{remark}
\subsection*{Proof of Theorem~\ref{thm:disks}}~\\
Let $f:\mathbb{CH}^{1} \hookrightarrow \T_{g,n}$ be a totally geodesic
isometry. Applying Lemma~\ref{lem-real-analytic}, we deduce that $f$
has a \textit{unique} holomorphic extension in a neighborhood of the
totally real diagonal
$\delta(\mathbb{CH}^{1})\subset\mathbb{CH}^{1}\times\overline{\mathbb{CH}^{1}}$. We
will show that $f$ extends to a
holomorphic map from $\mathbb{CH}^{1}\times\overline{\mathbb{CH}^{1}}$ to $\T_{g,n}$.~\\
We start by defining a \textit{new} map
$F:\mathbb{CH}^{1}\times\overline{\mathbb{CH}^{1}} \rightarrow
\T_{g,n}$, satisfying:

1. $F(z,z)=f(z)$ for all $z\in\Delta \cong \mathbb{CH}^{1}$.

2. $F|_{\Gamma_r}$ is the \textit{unique} holomorphic extension of
$f|_{\text{Fix}(r)}$ for all $r\in\mathcal{R}$.
~\\
Let $r\in\mathcal{R}$ be a reflection. There is a \textit{unique}
(holomorphic) Teichm\"uller disk
$\phi_r:\mathbb{CH}^{1}\hookrightarrow\T_{g,n}$ such that the
intersection
$\phi_r(\mathbb{CH}^1)\cap f(\mathbb{CH}^{1})\subset \T_{g,n}$
contains the Teichm\"uller geodesic $f(\text{Fix}(r))$ and
$\phi_r(z)=f(z)$ for all $z\in\text{Fix}(r)$.

We define $F$ by $F(z,r(z))=\phi_r(z)$ for $z\in\mathbb{CH}^{1}$ and
$r\in\mathcal{R}$; equation~(\ref{eq:leaves}) shows that $F$ is
well-defined and satisfies conditions (1) and (2) above.

We \textit{claim} that
$F:\mathbb{CH}^{1}\times\overline{\mathbb{CH}^{1}}\rightarrow
\T_{g,n}$
is the \textit{unique} holomorphic extension of
$f: \mathbb{CH}^1 \hookrightarrow \T_{g,n}$ such that $F(z,z)=f(z)$
for $z \in\mathbb{CH}^1$.

\textit{Proof of claim}. We note that the restriction of $F$ on the
totally real diagonal $\delta(\mathbb{CH}^{1})$ agrees with $f$ and
that there is a \textit{unique} germ of holomorphic maps near
$\delta(\mathbb{CH}^{1})$ whose restriction on
$\delta(\mathbb{CH}^{1})$ coincides with $f$. Let us fix an element of
this germ $\tilde{F}$ defined on a neighborhood
$U\subset\mathbb{CH}^{1}\times\overline{\mathbb{CH}^{1}}$ of
$\delta(\mathbb{CH}^{1})$. For every $r\in\mathcal{R}$, the
restrictions of $F$ and $\tilde{F}$ on the intersection
$U_r= U \cap \Gamma_r$ are holomorphic and equal along the
real-analytic arc $U_r \cap \delta(\mathbb{CH}^{1}) \subset U_r$;
hence they are equal on $U_r$. Since
$\mathbb{CH}^{1}\times\overline{\mathbb{CH}^{1}} =
\bigcup_{r\in\mathcal{R}}\Gamma_r$,
we conclude that $F|_{U}=\tilde{F}$ and, in particular, $F$ is
holomorphic near the totally real diagonal $\delta(\mathbb{CH}^{1})$.
% ----------------------------------------------------------------
% --------------------- REFORMULATE THIS -------------------------
% ----------------------------------------------------------------
Since, in addition to that, $F$ is holomorphic along all the leaves
$\Gamma_r$ of the foliation, we deduce~\footnote{For a simple proof of
  this claim using the power series expansion of $F$ at
  $(0,0)\in\mathbb{CH}^{1}\times\overline{\mathbb{CH}^{1}}$,
  see~\cite[Lemma~2.2.11]{Hormander:book}.} that it is holomorphic at
all points of $\mathbb{CH}^{1}\times\overline{\mathbb{CH}^{1}}$.\qed
% ----------------------------------------------------------------
% ----------------------------------------------------------------
% ----------------------------------------------------------------

In order to finish the proof of the theorem, we use the \textit{key}
observation~(\ref{eq:flip}); which we recall as follows: the points
$(z,w)$ and $(w,z)$ are always contained in the same leaf $\Gamma_r$
of the foliation for all $z,w\in\Delta\cong\mathbb{CH}^{1}$. Using the
fact that the restriction of $F$ on every leaf $\Gamma_{r}$ is a
Teichm\"uller disk, we conclude that
$d_{\T_{g,n}}(F(z,w),F(w,z))=d_{\mathbb{CH}^{1}}(z,w)$.
  
Let $\theta \in \mathbb{R}/2\pi\mathbb{Z}$, it follows that at least
one of $\displaystyle F(\rho e^{i\theta},0)$ and
$\displaystyle F(0,\rho e^{i\theta})$ diverges in Teichm\"uller space
as $\rho \rightarrow 1$. In particular, there is a subset
$I\subset \mathbb{R}/2\pi\mathbb{Z}$ with positive measure such that
either $F(\rho e^{i\theta},0)$ or
$\displaystyle F(0,\rho e^{i\theta})$ diverges as $\rho \rightarrow 1$
for all $\theta \in I$.

We assume first that the former of the two is true. Using that
$F: \mathbb{CH}^1\times\overline{\mathbb{CH}^1} \rightarrow \T_{g,n}$
is holomorphic, we deduce from~\cite{Tanigawa:holomap} (Sullivan's
rigidity theorem) that the family $\{F(z,\overline{w})\}_{w\in\Delta}$
of holomorphic maps
$ F(\cdot,\overline{w}): \Delta\cong\mathbb{CH}^1 \rightarrow
\T_{g,n}$
for $w\in \Delta\cong\mathbb{CH}^1$ is \textit{trivial}. Therefore,
$F(z,0)=F(z,z)=f(z)$ for all $z\in \Delta$ and, in particular, $f$ is
holomorphic. If we assume that the latter of the two is true we
similarly deduce that $F(0,z)=F(z,z)=f(z)$ for all $z\in \Delta$ and,
in particular, $f$ is anti-holomorphic.\qed

% ----------------------------------------------------------------
% ----------------------------------------------------------------

\section{The class of disk-rigid domains}\label{sec:disk-rigid}
In this section we formulate a general theorem that applies to a large class of bounded domains,
which we apply to deduce Corollary~\ref{cor:intro} and Theorem~\ref{thm:disks:other}.~\\

Let $\mathcal{B}\subset\C^N$ be a bounded domain and $f:\Delta\rightarrow \mathcal{B}$ a holomorphic
map. We call the map $f$ \textit{essentially proper} if $\partial\Delta \setminus B_f$ has positive
(Lebesgue) measure, where $B_f$ denotes the set of \textit{bounded rays}, ie.
$B_f = \{~ e^{i\theta}\in\partial\Delta ~:~ \sup_{t\in[0,1)}d_{\mathcal{B}}(f(0),f(te^{i\theta})) <
+\infty ~\}$.

\begin{defn}\label{defn:disk-rigid}
  A bounded domain $\mathcal{B}\subset\C^{N}$ is \textit{disk-rigid}, if it satisfies:
  \begin{enumerate}[leftmargin=3em]
  \item every unit-speed geodesic $\gamma: \R \hookrightarrow \mathcal{B}$, for the Kobayashi
    metric, extends to a complex geodesic
    $\gamma_{\C}: \Delta\cong\mathbb{CH}^1 \hookrightarrow \mathcal{B}$ such that
    $\gamma(t) = \gamma_{\C} (\text{tanh}(t))$, for $t \in \R$,
  \item every holomorphic family $\{f_t\}_{t\in\Delta}$ of holomorphic maps
    $f_t : \Delta\cong\mathbb{CH}^1 \rightarrow \mathcal{B}$, with $f_0$ an \textit{essentially
      proper} map, is \textit{trivial} ie. $f_t = f_0$ for all $t\in\Delta$.
  \end{enumerate}
\end{defn}

\subsection*{Examples} 
 \begin{enumerate}[leftmargin=2em]
 \item Teichm\"uller spaces $\T_{g,n}$ of finite dimension are \textit{disk-rigid}. See
   ~\S~\ref{sec:prelim}, Theorems~\ref{thm:slodkowski},~\ref{thm:sullivan}.~\\

 \item The bi-disk $\mathbb{CH}^1\times\overline{\mathbb{CH}^1}$ is a convex domain that is not
   \textit{disk-rigid}. A bounded symmetric domain $\mathcal{B}\subset\C^N$ is \textit{disk-rigid}
   if and only if it has rank one: $\mathcal{B}\cong \mathbb{CH}^N$.~\\

 \item All \textit{strictly} convex bounded domains $\mathcal{B}\subset \C^N$ are
   \textit{disk-rigid}. We recall that a domain $\mathcal{B}\subset \C^N$ is \textit{strictly}
   convex if $\{~t\cdot P + (1-t)\cdot Q ~:~ t\in(0,1) ~\} \subset \mathcal{B}$ for every pair of
   distinct points $P\not =Q$ in the closure $\overline{\mathcal{B}} \subset
   \C^N$. See~\cite{Nikolov:Pflug:Zwonek:Kobayashi:convex}~\\
 \end{enumerate}

 The proof of Theorem~\ref{thm:disks} in \S~\ref{sec:disks} used only those features of $\T_{g,n}$
 captured in the definition of a \textit{disk-rigid} domain. In particular, the following result
 follows as well.
 \begin{thm} \label{thm:disk-rigid-holo} 
   Let $\mathcal{B}\subset\C^N$ be a \textit{disk-rigid} domain. Every totally geodesic isometry
   $f: \mathbb{CH}^1 \hookrightarrow \mathcal{B}$ for the Kobayashi metric is either holomorphic or
   anti-holomorphic.
 \end{thm}

 We also have the following generalisation, which implies Corollary~\ref{cor:intro} and
 Theorem~\ref{thm:disks:other}. The proof follows from Theorem ~\ref{thm:disk-rigid-holo} and Weyl's
 regularity lemma.

 \begin{thm}\label{thm:disk-convex}
   Let $\mathcal{B}_1$, $\mathcal{B}_2$ be two complete \textit{disk-rigid} domains for the
   Kobayashi metric. Every totally geodesic isometry
   $f: \mathcal{B}_1 \hookrightarrow \mathcal{B}_2$ is either holomorphic or anti-holomorphic.
 \end{thm}
\begin{proof}
  % For any fixed Riemann surface $X \in \T_{g,n}$, we choose coordinates given by the Bers embeddings based at the points
  % $X$ and $f(X)$. Nehari's bound and the non-expanding property of the Teichm\"uller-Kobayashi metric implies that the
  % map $f$ is C-Lipschitz near the point $X$, for a universal constant $C$.~\cite{McMullen:kahler} Hence, it is
  % differentiable at almost all points near $X$. Since $X$ was arbitrary, we conclude that $f$ is differentiable at
  % almost all points of $\T_{g,n}$.
  
  In a sufficiently small neighborhood of a point, the Kobayashi metric is bi-Lipschitz to a
  Hermitian metric.~\cite{Kobayashi:book:hyperbolic} It follows that a totally geodesic isometry
  $f: \mathcal{B}_1 \hookrightarrow \mathcal{B}_2$ is locally Lipschitz and hence it is
  differentiable at almost all points of $\mathcal{B}_1$, by Rademacher's theorem (see Theorem 3.1.6
  in~\cite{Federer:book:geometric}).

  Let $p\in \mathcal{B}_1$ such that the (real) linear map $df_p : T_p\mathcal{B}_1 \rightarrow T_p\mathcal{B}_2$ exists.
  Using Theorem~\ref{thm:disk-rigid-holo}, we conclude that $f$ sends complex geodesics in $\mathcal{B}_1$ through $p$ to
  complex geodesics in $\mathcal{B}_2$ through $f(p)$ and, in particular, the linear map $df_p$ sends complex lines in
  $T_p\mathcal{B}_1$ to complex lines in $T_p\mathcal{B}_2$. We conclude that $df_p$ is either a complex linear map or
  complex anti-linear map.
  
  The assumption that the Kobayashi metric of $\mathcal{B}_1$ and $\mathcal{B}_2$ is complete
  implies that there is a complex geodesic between any pair of distinct points in $\mathcal{B}_1$
  and $\mathcal{B}_2$. Hence, $df_p$ is either complex linear for almost every
  $p\in \mathcal{B}_1$ or complex anti-linear for almost every $p\in \mathcal{B}_1$. In
  particular, up to conjugation, $f$ is holomorphic as a distribution and the theorem follows from
  Weyl's regularity lemma.~\cite{Krantz:book}
\end{proof}

% ----------------------------------------------------------------

%%% Local Variables:
%%% TeX-master: "main"
%%% End:

% ----------------------------------------------------------------